\documentclass[11pt, reqno]{amsart}

\usepackage[margin=1in]{geometry}

\usepackage[utf8]{inputenc}
\usepackage{import}
\usepackage{amsmath,amssymb,amsthm,amsfonts}
\usepackage{physics}
\usepackage[pdfborder={0 0 0}]{hyperref}
\usepackage{mathtools}
\usepackage{cases}
\usepackage{xcolor}
\usepackage{enumerate}

\newcommand{\Ni}{\noindent}

\newcommand{\vpran}[1]{\left(#1\right)}

\newcommand{\One}{{\bf 1}}
\newcommand{\RR}{\mathbb{R}}
\newcommand{\nn}{\nonumber}

\numberwithin{equation}{section}

\title[$L^p$-norms]{$L^p$-norms for the homogeneous non-cutoff Boltzmann equation with soft potentials}

\author{Matt Spragge}
\address{Department of Mathematics, Simon Fraser University, 8888 University Dr., Burnaby, BC V5A 1S6, Canada}
\email{matt\_spragge@sfu.ca}
\author{Weiran Sun}
\address{Department of Mathematics, Simon Fraser University, 8888 University Dr., Burnaby, BC V5A 1S6, Canada}
\email{weiran\_sun@sfu.ca}

\date{}

\DeclarePairedDelimiter\jap{\langle}{\rangle}

\newcommand{\gam}{\gamma}

\newcommand{\one}[1]{\chi_{\left\{#1\right\}}}

\newcommand{\rn}[1]{\mathbb{R}^{#1}}
\newcommand{\sn}[1]{\mathbb{S}^{#1}}
\newcommand{\wlp}[2]{L^{#1}_{#2}}
\newcommand{\lp}[1]{L^{#1}}
\newcommand{\whp}[2]{H^{#1}_{#2}}
\newcommand{\hp}[1]{H^{#1}}

\newtheorem{theorem}{Theorem}[section]
\newtheorem{lemma}{Lemma}[section]
\newtheorem{remark}{Remark}
\newtheorem{corollary}{Corollary}[section]

\begin{document}

\maketitle

\begin{abstract}
We establish a priori estimates showing the propagation and generation of $L^p$-norms for solutions to the non-cutoff spatially homogeneous Boltzmann equation with soft potentials. The singularity of the collision kernel is key to generate regularization and inhomogeneity in the energy estimates of the $L^p$-norms. Our result extends \cite{Alo19} from the hard potential cases to the soft ones.
    
\end{abstract}

\section{Introduction}\label{sec;introduction}

In this paper we show the propagation and generation of $L^p$-norms for solutions to the non-cutoff spatially homogeneous Boltzmann equation with soft potentials. The Cauchy problem for this equation is given by 
\begin{equation}\label{eq:Boltzmann}
    \begin{cases}
        \partial_t f(t,v)=Q(f,f)(t,v) & (t,v)\in(0,\infty)\times\rn{3}, \\
        f(0,v)=f_0(v),
    \end{cases}
\end{equation}
where $f(t,v)$ is the nonnegative density function at velocity $v\in\rn{3}$ and time $t>0$. The Boltzmann collision operator, $Q$, is defined by
\begin{equation*}
    Q(g,f)\coloneqq\int_{\rn{3}}\int_{\sn{2}}(g_*'f'-g_*f)B(v-v_*,\sigma)d\sigma\, dv_*,
\end{equation*}
where $(v', v'_\ast)$ and $(v, v_\ast)$ are two pairs of velocities before and after the elastic collision or vice versa. They satisfy the conservation of momentum and energy, which gives, for some $\sigma \in \sn{2}$, 
\begin{equation*}
    v'=\frac{v+v_*}{2}+\frac{\abs{v-v_*}}{2}\sigma,\quad v_*'=\frac{v+v_*}{2}-\frac{\abs{v-v_*}}{2}\sigma. 
\end{equation*}
For brevity we write 
\begin{align*}
   f\coloneqq f(t,v), 
\quad
   f_*\coloneqq f(t,v_*),
\quad 
   f'\coloneqq f(t,v'),
\quad
   f_*'\coloneqq f(t,v_*').
\end{align*} 
The Boltzmann collision kernel, $B$, is assumed to have the form
\begin{equation}\label{cond:B-1}
    B(v-v_*,\sigma) 
 = \Phi(\abs{v-v_*}) \, b\left(\frac{v-v_*}{\abs{v-v_*}}\cdot\sigma\right), 
\end{equation}
where the angular collision kernel, $b$, possesses a nonintegrable singularity described by
\begin{equation}\label{cond:B-2}
    s\in(0,1), 
\quad
    \sin\theta \, b(\cos\theta) 
\approx 
   \frac{b_0}{\theta^{1+2s}}
\quad
\text{for $\theta$ near  0},
\quad
\cos\theta=\frac{v-v_*}{\abs{v-v_*}}\cdot\sigma.
\end{equation}
Moreover, we assume that $\Phi$ takes the form
\begin{equation}\label{cond:B-3}
    \Phi(\abs{v-v_*}) = \abs{v-v_*}^\gam,\quad \gam\in\left(-3,0\right).
\end{equation}
This form of collision kernel applies, for example, to particles with inverse-power law potentials of the form $r^{-\ell}$ with $r$ denoting the distance between two particles. In this case, $\gam$ and $s$ may be expressed as
\begin{equation}\label{eqn;intro;inverse-power law gamma s relationship}
    \gam=\frac{\ell-4}{\ell},
\qquad 
   s=\frac{1}{\ell}. 
\end{equation}

Through a variety of change of variables, we recall that the Boltzmann collision operator admits the following weak formulations:
\begin{align} \label{Q:weak-form}
& \quad \,
    \int_{\rn{3}}Q(g,f)\varphi dv   \nn
\\
&= -\frac{1}{4}\int_{\rn{3}\times\rn{3}}\int_{\sn{2}}(\varphi'+\varphi'_*-\varphi-\varphi_*)(g_*'f_*-g_*f)\abs{v-v_*}^\gam 
 b\left(\frac{v-v_*}{\abs{v-v_*}}\cdot\sigma\right)d\sigma\, dv_*\, dv,
\end{align}
and
\begin{equation}\label{eqn:weak-form}
    \int_{\rn{3}}Q(g,f)\varphi dv=\int_{\rn{3}\times\rn{3}}\int_{\sn{2}}(\varphi'-\varphi)g_*f\abs{v-v_*}^\gam b\left(\frac{v-v_*}{\abs{v-v_*}}\cdot\sigma\right)d\sigma\, dv_*\, dv,
\end{equation}
where $\varphi(v)$ is a smooth test function. As a consequence of~\eqref{Q:weak-form}, we have the following conservation laws for the Boltzmann equation~\eqref{eq:Boltzmann}:
\begin{equation}\label{eqn;intro;conservation laws}
    \dv{t}\int_{\rn{3}\times\rn{3}}f\mqty(1 \\ v \\ \abs{v}^2)dv=0.
\end{equation}
It also follows from~\eqref{Q:weak-form} that $f$ satisfies
\begin{equation}\label{eqn;intro;H theorem}
    \dv{t}\int_{\rn{3}}f\log fdv\leq0,
\end{equation}
which, in particular, is a result famously known as \textit{Boltzmann's $H$ theorem} (see, for example, \cite{Vil02}). Equations (\ref{eqn;intro;conservation laws}) and (\ref{eqn;intro;H theorem}) indicate that we may take the class:
\begin{equation*}\label{eqn;intro;U definition}
    \mathcal{U}(D_0,E_0)
 \coloneqq
    \left\{f\in\lp{1}(\rn{3}):f \geq 0, \,\, \norm{f}_{\lp{1}}\geq D_0, \,\,\int_{\rn{3}}f(1+\abs{v}^2+\log(1+f))dv\leq E_0\right\},
\end{equation*}
for some positive constants $D_0$ and $E_0$, as the minimal function space for the Cauchy problem~\eqref{eq:Boltzmann}. It is therefore commonly assumed---see, for example, \cite{Vil98,CarlCarvL09,AMUXY12,P-CT18,Alo19}---that the initial data $f_0$ in~\eqref{eq:Boltzmann} is contained at least in $\mathcal{U}(D_0,E_0)$. 

To briefly remark on notation, we note that we use $L^p(\RR^3), L^\infty(\RR^3), H^s(\RR^3)$ as the usual Lebesgue and Sobolev spaces. 
We also frequently use the following weighted spaces:
\begin{align*}
    \wlp{p}{w}(\RR^3) &\coloneqq \left\{f:\norm{f}_{\wlp{p}{w}(\RR^3)}\coloneqq\norm{\jap{\cdot}^w f}_{\lp{p}(\RR^3)}<\infty\right\}
\qquad
   \jap{\cdot}\coloneqq\sqrt{1+\abs{\cdot}^2}
    \intertext{and}
    \whp{s}{r}(\RR^3) &\coloneqq \left\{f:\norm{f}_{\whp{s}{r}(\RR^3)}\coloneqq\norm{\jap{\cdot}^rf}_{\hp{s}(\RR^3)}<\infty\right\}.
\end{align*}

The main purpose of this work is to address the $\lp{p}$ generation and propagation problems for solutions to~\eqref{eq:Boltzmann} with soft potentials and without angular cutoff. Such estimates have been established in~\cite{Alo19} for the hard potentials and our main contribution is to extend them to the soft potentials. In particular, we provide a priori $\lp{p}$ estimates for solutions to~\eqref{eq:Boltzmann} with suitable initial data $f_0$ not necessarily belonging to $\lp{p}$ (see Theorems~\ref{thm:Lp-weak-soft}-\ref{thm:L-p-strong-soft} for $p < \infty$ and Theorems~\ref{thm:L-infty-strong}-\ref{thm-Linfty-weak-soft} for $p = \infty$). Note that similar to~ \cite{Alo19}, our a priori estimates are derived based on the weak formulation of the equation rather than its local strong form. Hence they will be suitable tools to study weak solutions to the Boltzmann equation.

Under an angular cutoff assumption, the propagation of $\lp{p}$-norms for solutions to the Boltzmann equation has been studied by many authors (see for example \cite{Ark72, Ark83, Gus88, MV04, Wen95}). 
There are less results regarding the propagation/generation of the $L^p$-norms for the Boltzmann equation without angular cutoff, especially for the regime of very soft potentials where $\gamma < - 2s$ (see \cite{Alo19, Silvestre2016, GIS2023}).
The most relevant work for us is \cite{Alo19}, in which a priori estimates are established to show the generation and propagation of $\lp{p}$-norms for $p\in(1,\infty]$ and $\gam\geq0$. In \cite{Alo19} it is shown that the energy estimates for $\norm{f}_{L^p}$ can be closed by using
the cancellation lemma \cite{ADVW00} and the coercivity estimate in \cite{AMUXY12} when $p \in (1, \infty)$. The same strategy cannot entirely be imitated when $p=\infty$. Instead, the $\lp{\infty}$ results are obtained in \cite{Alo19} using an adaptation of the De Giorgi level set method.

In this paper we take a similar approach to \cite{Alo19}. 
The difficulties for the soft potential lie in the extra singularity in the collision kernel $|v - v_\ast|^\gamma$ when $v -  v_\ast$ is close to zero. Moreover, contrary to the hard potential where $|v - v_\ast|^\gamma$ generates moments, one loses moments when $\gamma < 0$. As a consequence, we assume that the initial data $f_0$ in~\eqref{eq:Boltzmann} belongs to $\mathcal{U}(D_0,E_0)\cap\wlp{1}{w}$, for some sufficiently large weight $w$, rather than the usual $\mathcal{U}(D_0,E_0)$ alone. The global $L^1_w$-norm propagation shown in \cite{CarlCarvL09} is essential to overcome the difficulty of the moment loss. We would like to comment that upon the completion of our work, we discovered the paper \cite{GIS2023}, where some basic estimates for their approximate $L^p$-norms share a similar spirit to our work.

Solutions in the $L^p$-settings are also extensively studied for the Landau equation (see for example \cite{ABDL2022, ABDL2023, GGL2024, GL2024-1, Silvestre2017} and the references therein). For instance, in~\cite{GL2024-1}, it has been shown (among other results) that $L^p$-norms propagate for the Landau-Coulomb equation for $p > 3/2$. Note that our parameter range of $p >  \frac{3}{3+\gamma+2s}$ when $\gamma < -2s$ is consistent with the Landau equation where $p > 3/2$. It is well known that the Landau equation can be viewed as a grazing limit of the Boltzmann equation as $s \to 1$. In this sense, parameters in the Landau-Coulomb equation corresponds to $(\gamma, s) = (-3, 1)$ which matches with our range of $p > \frac{3}{3+\gamma+2s}$ in $\RR^3$. The result in~\cite{GL2024-1} is recently improved to the boundary point where $p= 3/2$ (see ~\cite{GGL2024}). It will be interesting to investigate whether the borderline case of $p = \frac{3}{3+\gamma+2s}$ holds for the Boltzmann equation with very soft potentials where $\gamma < -2s$.

\section{$\lp{p}$ Theory}\label{sec;Lp theory} We begin this section by recalling several relevant estimates in the literature regarding the homogeneous Boltzmann equation. 

\begin{lemma}[\cite{Alo19}, Section 2] \label{lem:Q-I-J-p}
For any $g \in \mathcal{U}(D_0,E_0)$ and $1 < p < \infty$, denote
\begin{align}
    I_p(g,f) &\coloneqq \int_{\rn{3}\times\rn{3}}\int_{\sn{2}}g_*\left[(f')^p-f^p\right]B(v-v_*,\sigma)d\sigma\, dv_*\, dv,\label{def:I-p}
\\
    J_p(g,f) &\coloneqq \int_{\rn{3}\times\rn{3}}\int_{\sn{2}}g_*\left[(f')^{\frac{p}{2}}-f^{\frac{p}{2}}\right]^2B(v-v_*,\sigma)d\sigma\, dv_*\, dv,\label{def:J-p}
\end{align}
where $B$ satisfies~\eqref{cond:B-1}-\eqref{cond:B-3} for any $\gam\in(-3,1)$. Then the collision operator satisfies
\begin{equation}\label{bound:Q-g-f-general}
    \int_{\rn{3}}Q(g, f)f^{p-1}dv
\leq
   \frac{1}{p'}I_p(g, f)-\frac{1}{\max\{p,p'\}}J_p(g, f),
\qquad
  \frac{1}{p} + \frac{1}{p'} = 1.
\end{equation}
\end{lemma}

The estimate of $J_p$ is given in \cite[Proposition 2.1]{AMUXY12} (see \cite[equation (2.9)]{AMUXY12}):
\begin{lemma}[\cite{AMUXY12}]  \label{lem:J-p}
Suppose $g \in \mathcal{U}(D_0,E_0)$ and $1 < p < \infty$. Then $J_p$ defined in~\eqref{def:J-p} satisfies
\begin{equation}\label{bound:J-p}
    J_p(g, f)
\geq 
   C_0\norm{f^{\frac{p}{2}}}_{\whp{s}{\gam/2}}^2
  - C_1\norm{f^{\frac{p}{2}}}_{\wlp{2}{\gam/2}}^2
\end{equation}
for constants $C_0,C_1$ depending only on $D_0,E_0$. 
\end{lemma}

We will also make use of (2.4) from \cite{AMUXY12}:
\begin{lemma}[\cite{AMUXY12}] \label{lem:I-2-AMUXY12}
Suppose $B$ satisfies~\eqref{cond:B-1}-\eqref{cond:B-3}. Then $I_p$ defined in~\eqref{def:I-p} with $p=2$ satisfies
\begin{align*} 
   I_2 (g, f)
\leq 
  C \int_{\rn{3}\times\rn{3}} \abs{g_*} f^2\abs{v-v_*}^\gam dv_*\, dv
\leq
  C \norm{g}_{L^1_{|\gamma|}}
  \norm{f}_{H^{|\gamma|/2}_{\gamma/2}}^2,
\end{align*}
where $C$ depends on $\gamma, s$.
\end{lemma}

By replacing $f$ in Lemma~\ref{lem:I-2-AMUXY12} with $f^{p/2}$, we immediately have
\begin{corollary} \label{cor:I-p}
Suppose $B$ satisfies~\eqref{cond:B-1}-\eqref{cond:B-3}. Then $I_p$ defined in~\eqref{def:I-p} satisfies
\begin{align*} 
   I_p (g, f)
\leq 
  C \int_{\rn{3}\times\rn{3}} \abs{g_*} f^p \abs{v-v_*}^\gam dv_*\, dv
\leq
  C \norm{g}_{L^1_{|\gamma|}}
  \norm{f^{p/2}}_{H^{|\gamma|/2}_{\gamma/2}}^2,
\end{align*}
where $p \in (1, \infty)$ and $C$ depends on $\gamma, s$. 
\end{corollary}

We will make use of the weighted-$L^1$ propagation of the solution as shown in~\cite[Theorem 1]{CarlCarvL09}:
\begin{theorem} [\cite{CarlCarvL09}] \label{thm:L1-prop}
Suppose $w > 2$ and $B$ satisfies~\eqref{cond:B-1}-\eqref{cond:B-3} with $\gam\in(-3,1)$. Then
\begin{equation}\label{prop:L1}
    \norm{f(t)}_{\wlp{1}{w}} \leq C_w(1+t),\quad \text{for all } t>0.
\end{equation}
\end{theorem}

For the convenience of the reader, we also recall two classical inequalities. 
\begin{lemma} (Hardy) \label{lem;Lp;Hardy's inequality}
    Let $\ell\in(0,1)$ and $d>2\ell$. Then, for any $F:\rn{d}\to\mathbb{R}$ in $\hp{\ell}(\rn{d})$,
    \begin{equation*}
        \sup_{v_*}\int_{\rn{d}}\frac{F(v)^2}{\abs{v-v_*}^{2\ell}}dv \leq C\norm{F}_{\hp{\ell}}^2,
\qquad
\text{for some $C>0$.}
    \end{equation*}
\end{lemma}

\begin{lemma} (Hardy-Littlewood-Sobolev)
Suppose $0 < \alpha < d$ and $1 < p < q < \infty$. Then 
\begin{align*}
   \norm{f \ast |\cdot|^{\alpha - d}}_{L^q(\rn{d})} \leq C \norm{f}_{L^p(\rn{d})}, 
\qquad
  \frac{1}{q} = \frac{1}{p} - \frac{\alpha}{d}. 
\end{align*}
\end{lemma}

In what follows, we often make use of the Sobolev embedding:
\begin{align} \label{Sobolev}
  \norm{f}_{L^{p_s}_{\gamma/2}(\RR^3)}
\leq
  C \norm{f^{p/2}}_{H^s_{\gamma/2}(\RR^3)}^{2/p}, 
\qquad
  p_s = \frac{3p}{3-2s}.
\end{align}

\medskip
Our first set of estimates is to bound $I_p$, which will be the main ingredient in bounding the collision term in the energy estimates. 
\begin{lemma} \label{lem: Ip-bound}
Let $B$ satisfy~\eqref{cond:B-1}-\eqref{cond:B-3}, $f$ be sufficiently smooth, and $I_p$ be as in~\eqref{def:I-p}. Then
\begin{enumerate}[i)]
\item for $\gam\in(-2s,0)$ and $p\in(1,\infty)$,  
\begin{equation}\label{bound:Ip-weak-soft}
  I_p(g,f) 
\leq 
  C \norm{g}_{L^1_{|\gamma|}}
  \norm{f^{\frac{p}{2}}}_{H^{|\gamma|/2}_{\gamma/2}}^2,
\end{equation}
where $C$ depends on $\gamma, s$;

\item for $\gam\in(-3,-2s]$ and $p\in\left(\frac{3}{3+\gamma+2s},\frac{3}{3+\gam}\right)$, there exist a constant $C$, a weight $w$, and $\theta_1 \in (0, 1)$, all explicitly computable, such that
\begin{equation}\label{bound:Ip-strong-1}
   I_p(g,f) 
\leq
  C \norm{g}_{\lp{p}} 
  \norm{f}_{L^1_w}^{p \theta_1} \norm{f^{\frac{p}{2}}}_{H^s_{\gamma/2}}^{2(1-\theta_1)};
\end{equation}

\item for $\gam\in(-3,-2s]$ and $p\in\left(\frac{3}{3+\gamma+2s},\infty\right)$, if in addition, $g\in\lp{p_0}$ for some $p_0\in\left(\frac{3}{3+\gamma+2s}, \,\, \frac{3}{3+\gamma}\right)$, 
then there exist a constant $C$, a weight $w$, and $\theta_2 \in (0, 1)$, all explicitly computable, such that
\begin{equation}\label{bound:Ip-strong-2}
    I_p(g,f) 
\leq 
  C \norm{g}_{\lp{p_0}}
  \norm{f}_{L^1_w}^{p \theta_2} \norm{f^{\frac{p}{2}}}_{H^s_{\gamma/2}}^{2(1-\theta_2)}.
\end{equation}
\end{enumerate}
The constants $C$ in~\eqref{bound:Ip-strong-1}-\eqref{bound:Ip-strong-2} only depend on $\gamma, s, p, p_0$. In particular, they are independent of~$f, g$.
\end{lemma}
\begin{proof}
{\it i)} The inequality in~\eqref{bound:Ip-weak-soft} follows from Corollary~\ref{cor:I-p} or Lemma~\ref{lem:I-2-AMUXY12}. 

\smallskip
\Ni {\it ii)} The first inequality in Corollary~\ref{cor:I-p} together with the H{\"o}lder's and HLS inequalities gives
\begin{equation*}
  I_p(g,f) 
\leq 
  C \norm{g}_{\lp{p}}\norm{f^p*\abs{\cdot}^{\gam}}_{\lp{p'}}
\leq
  C\norm{g}_{\lp{p}}\norm{f^p}_{\lp{q}}, 
\qquad
  \frac{1}{p} + \frac{1}{p'} = 1, 
\qquad
  \frac{1}{q} = 2+\frac{\gam}{3}-\frac{1}{p}. 
\end{equation*}
Note that $p<\frac{3}{3+\gamma}$ implies $q>1$. Using the H{\"o}lder's inequality and~\eqref{Sobolev}, we have
\begin{align*}
  \norm{f^p}_{\lp{q}}
= \norm{f}_{L^{pq}}^p
\leq
   \norm{f}_{L^1_w}^{p \theta_1} \norm{f}_{L^{p_s}_{\gamma/2}}^{p(1-\theta_1)}
\leq
  \norm{f}_{L^1_w}^{p \theta_1} \norm{f^{\frac{p}{2}}}_{H^s_{\gamma/2}}^{2(1-\theta_1)}, 
\end{align*}
where the weight $w$ is large enough and 
\begin{align*}
  p_s = \frac{3p}{3-2s}, 
\qquad
  \frac{1}{pq} = \theta_1 + \frac{1 - \theta_1}{p_s}.  
\end{align*}
The interpolation of $L^{pq}$ into $L^1_w \cap L^{p_s}_{\gamma/2}$ holds since by the assumption that $p > \frac{3}{3+\gamma+2s}$ and the definition of $q$, we have
\begin{align*}
  1 < pq = \frac{p}{2 + \frac{\gamma}{3} -\frac{1}{p}} < \frac{3 p}{3-2s} = p_s.
\end{align*}
Altogether, we have 
\begin{align*}
   I_p
\leq 
  C\norm{g}_{\lp{p}} 
  \norm{f}_{L^1_w}^{p \theta_1} \norm{f^{\frac{p}{2}}}_{H^s_{\gamma/2}}^{2(1-\theta_1)}.
\end{align*}

\Ni {\it iii)} Similar to the argument in {\it ii)}, we have
\begin{align*} 
   I_p(g,f) 
\leq 
   \norm{g}_{\lp{p_0}}\norm{f^p*\abs{\cdot}^\gam}_{\lp{p_0'}} \nn
\leq
   C\norm{g}_{\lp{p_0}}\norm{f^p}_{\lp{q_0}}, 
\qquad
  \frac{1}{p} + \frac{1}{p'} = 1, 
\qquad
   \frac{1}{q_0}=2+\frac{\gam}{3}-\frac{1}{p_0}. 
\end{align*}
The range of $p_0$ guarantees that $q_0 > 1$ and
\begin{align*}
   1 < p q_0 = \frac{p}{2 + \frac{\gamma}{3} - \frac{1}{p_0}} < \frac{3p}{3 - 2s} = p_s, 
\qquad
  \forall p > 1. 
\end{align*}
Then, following the same interpolation procedure as in part {\it ii)}, we obtain
\begin{equation*}
  I_p
\leq
  C \norm{g}_{\lp{p_0}}
  \norm{f}_{L^1_w}^{p \theta_2} \norm{f^{\frac{p}{2}}}_{H^s_{\gamma/2}}^{2(1-\theta_2)}, 
\qquad
   \frac{1}{p q_0} = \theta_2 + \frac{1 - \theta_2}{p_s}.   \qedhere
\end{equation*}
\end{proof}

\subsection{$\lp{p}$ Generation for $\gam>-2s$}\label{sec;Lp;main results}

We now prove our first result of the generation of $L^p$-norms.
\begin{theorem}\label{thm:Lp-weak-soft}
Let $p\in(1,\infty)$, $D_0,E_0>0$, and $T>0$ be fixed. If $f \geq 0$ is a sufficiently smooth solution to~\eqref{eq:Boltzmann} on $(0,T]\times\rn{3}$ with $f_0 \in \mathcal{U}(D_0,E_0) \cap L^1_w$ where $w$ is large enough and $B$ satisfies~\eqref{cond:B-1}-\eqref{cond:B-3} with $\gam\in(-2s,0)$, then 

\Ni (a) (Generation) there are constants $C$ and $\alpha_1$ such that
\begin{equation} \label{generation-1}
    \norm{f(t)}_{\lp{p}}
+ \int_t^T \norm{f^{\frac{p}{2}}}_{\whp{s}{\gam/2}}^2(\tau) d\tau
\leq 
   C \left(t^{-\alpha_1}+1\right),\quad \text{for all } t\in(0,T];
\end{equation}

\Ni (b) (Propagation) if we further assume that $f_0 \in\mathcal{U}(D_0,E_0)\cap\lp{p}$, then there is a constant $C$ such ~that 
\begin{equation}\label{Lp-prop}
    \sup_{t\in[0,T]}\norm{f(t)}_{\lp{p}} 
\leq  
   C \norm{f_0}_{L^p}.
\end{equation}
The constants $C$ in parts (a)  and (b) depend at most on $T, p,\gam,s,D_0,E_0, \norm{f_0}_{L^1_w}$. 
\end{theorem}

\begin{proof}
Multiply~\eqref{eq:Boltzmann} by $pf^{p-1}$ and integrate over $v\in\rn{3}$. By Lemmas~\ref{lem:Q-I-J-p}, \ref{lem:J-p}, Theorem~\ref{thm:L1-prop} and \eqref{bound:Ip-weak-soft} in Lemma~\ref{lem: Ip-bound}, we find that
\begin{align} \label{bound:energy-Lp-weak-soft}
   \dv{t} \norm{f(t)}_{L^p}^p 
&= p \int_{\rn{3}} Q(f, f) f^{p-1} dv
\leq
   \frac{p}{p'}I_p(f, f)-\frac{p}{\max\{p,p'\}}J_p(f, f)  \nn
\\
& \leq
  -C_0 \norm{f^{\frac{p}{2}}}_{\whp{s}{\gam/2}}^2
  +C_1\norm{f^{\frac{p}{2}}}_{\wlp{2}{\gam/2}}^2
  + \frac{p}{p'} I_p(f, f)  \nn
\\
& \leq
  -C_0 \norm{f^{\frac{p}{2}}}_{\whp{s}{\gam/2}}^2
  +C_1\norm{f^{\frac{p}{2}}}_{\wlp{2}{\gam/2}}^2
  + C \norm{f}_{L^1_{|\gamma|}}
  \norm{f^{\frac{p}{2}}}_{H^{|\gamma|/2}_{\gamma/2}}^2  \nn
\\
& \leq
  -C_0 \norm{f^{\frac{p}{2}}}_{\whp{s}{\gam/2}}^2
  +C_1\norm{f^{\frac{p}{2}}}_{\wlp{2}{\gam/2}}^2
  + C (1+t)
  \norm{f^{\frac{p}{2}}}_{H^{|\gamma|/2}_{\gamma/2}}^2  \nn
\\
& \leq
  -\frac{C_0}{2} \norm{f^{\frac{p}{2}}}_{\whp{s}{\gam/2}}^2
  + C \norm{f}_{L^p}^p, 
\end{align}
where the last step follows from the interpolation $H^{|\gamma|/2}_{\gamma/2} \hookrightarrow H^s_{\gamma/2} \cap L^2_{\gamma/2}$ for $\gamma \in (-2s, 0)$.

If $f_0 \in L^p$ , then the bound in~\eqref{Lp-prop} follows from the Gronwall's inequality. Otherwise, let us proceed with the generation of the $L^p$-norm. If we take $\theta_3 \in (0,1)$ so that 
\begin{align} \label{def:theta-0}
  \frac{1}{p} = (1-\theta_3) + \frac{\theta_3}{p_s},
\end{align} 
then by taking $w$ large enough and using Theorem~\ref{thm:L1-prop} and~\eqref{Sobolev}, we get
\begin{align} \label{bound:L-p-H-s}
    \norm{f}_{\lp{p}}^p 
\leq \norm{f}_{L^1_w}^{p(1-\theta_3)}\norm{f}_{L^{p_s}_{\gam/2}}^{p\theta_3} 
 \leq 
    C (1+T)^{p(1-\theta_3)} \norm{f^{\frac{p}{2}}}_{H^s_{\gamma/2}} ^{2\theta_3}
\leq
  \epsilon \norm{f^{\frac{p}{2}}}_{H^s_{\gamma/2}} ^{2}
  + C_{\epsilon}, 
\end{align}
By~\eqref{bound:energy-Lp-weak-soft} and taking $\epsilon$ small enough, we have
\begin{align} \label{bound:L-p-generation-basic}
   \dv{t} \norm{f(t)}_{L^p}^p
+ \frac{C_0}{4} \norm{f^{\frac{p}{2}}}_{\whp{s}{\gam/2}}^2
\leq
  C.
\end{align}
The second inequality in~\eqref{bound:L-p-H-s} also gives
\begin{equation} \label{bound:lower-H-s} 
   \norm{f^{\frac{p}{2}}}_{\whp{s}{\gam/2}}^{2}
\geq
    \frac{\norm{f}_{L^p}^{\frac{p}{\theta_3}}}{C^{\frac{1}{\theta_3}} (1+T)^{p \frac{1-\theta_3}{\theta_3}}}. 
\end{equation}
Applying~\eqref{bound:lower-H-s} to~\eqref{bound:L-p-generation-basic}, we find that
\begin{equation*}
    \dv{t}\norm{f(t)}_{\lp{p}}^p + C_2 \frac{\norm{f(t)}_{\lp{p}}^{\frac{p}{\theta_3}}}{(1+T)^{p\frac{1-\theta_3}{\theta_3}}} 
\leq 
   C, \qquad \text{for all }  t \in (0,T].
\end{equation*}
That is, $\norm{f(t)}_{\lp{p}}^p$ is a sub-solution to the ordinary differential equation
\begin{equation} \label{ODE:generation-1}
    \dv{t}X(t) + C_2 \frac{X(t)^{\frac{1}{\theta_3}}}{(1+T)^{p\frac{1-\theta_3}{\theta_3}}} = C, \qquad \text{for all } t \in (0,T].
\end{equation}
Similar as in the proof of \cite[Theorem 1]{Alo19}, there exists a constant $C$ such that 
\begin{equation}
    X^*(t) 
\coloneqq 
   C \vpran{t^{-\frac{\theta_3}{1-\theta_3}}+1}
\end{equation}
is a super-solution to \eqref{ODE:generation-1}. Thus the parameter $\alpha_1$ in~\eqref{generation-1} satisfies $\alpha_1 = \frac{\theta_3}{1-\theta_3}$ with $\theta_3$ defined in~\eqref{def:theta-0}. 
Integrating~\eqref{bound:L-p-generation-basic} from $t$ to $T$ and applying the bound for $\norm{f}_{L^p}$ in~\eqref{generation-1}, we have
\begin{equation*}   
   \int_t^T \norm{f^{\frac{p}{2}}}_{\whp{s}{\gam/2}}^2(\tau) d\tau
\leq
   C T + \norm{f(t)}_{L^p}^p
\leq
  C \left(t^{-\alpha_1}+1\right),\quad \text{for all } t\in(0,T]. \qedhere
\end{equation*}
\end{proof}

\smallskip
\subsection{$\lp{p}$-Generation for $\gam\leq-2s$}
We now present our results for the propagation and generation of $\lp{p}$-norms when $\gam\in(-3,-2s]$.

\begin{theorem}\label{thm:L-p-strong-soft}
Let $\gam\in(-3,-2s]$, $D_0,E_0>0$, $B$ satisfy~\eqref{cond:B-1}-\eqref{cond:B-3} and $T>0$ be fixed. Let $f(t,v)$ be a sufficiently smooth solution to~\eqref{eq:Boltzmann} on $(0,T]\times\rn{3}$ with $f_0 \in \mathcal{U}(D_0,E_0)\cap\wlp{1}{w}$ for $w$ sufficiently large. Then,
\begin{enumerate}[i)]
\item (Generation) if $f_0\in\lp{p_0}$ for some $p_0 \in \left(\frac{3}{3+\gam+2s}, \frac{3}{3+\gam} \right)$, then for any $p \in \left(\frac{3}{3+\gam+2s}, \infty\right)$, there are constants $C, t_\ast$ depending on $\gam,s,D_0,E_0,p$ and $\norm{f_0}_{L^{p_0} \cap L^1_w}$ and $\alpha_2>0$ depending on $\gamma, s$ such that we have
\begin{equation}\label{generation:L-p-1}
   \norm{f(t)}_{\lp{p}} 
+ \int_{t}^{t_*}\norm{f^{\frac{p}{2}}(\tau)}_{\whp{s}{\gam/2}}^2 d\tau
\leq 
  C \left(t^{-\alpha_2}+1\right), \quad \text{for all }t\in(0,t_\ast];
\end{equation} 

\medskip

\item (Propagation)
for $p\in\left(\frac{3}{3+\gam+2s}, \infty\right)$, if $f_0\in\lp{p}$, then there are constants $C > 0$ and  $t_\ast > 0$ depending on $\gam,s,D_0,E_0,p, \norm{f_0}_{L^p}$ such that
\begin{equation}\label{bound:L-p-propag}
   \norm{f(t)}_{\lp{p}} 
\leq 
   C
\quad \text{for all }t\in(0,t_\ast].
\end{equation}

\end{enumerate}
\end{theorem}

\begin{proof}
{\it i)} Similar as in the proof of Theorem~\ref{thm:Lp-weak-soft}, we have the a priori energy estimate
\begin{align*}
   \dv{t} \norm{f(t)}_{L^p}^p 
&= p \int_{\rn{3}} Q(f, f) f^{p-1} dv
\leq
   \frac{p}{p'}I_p(f, f)-\frac{p}{\max\{p,p'\}}J_p(f, f) 
\\
& \leq
  \frac{p}{p'} I_p(f, f)
  -C_0 \norm{f^{\frac{p}{2}}}_{\whp{s}{\gam/2}}^2
  +C_1 \norm{f}_{L^p}^p.
\end{align*}
By \eqref{bound:Ip-strong-2} in Lemma~\ref{lem: Ip-bound}, Theorem~\ref{thm:L1-prop} and Young's inequality, taking $p = p_0$ we have
\begin{align*}
  I_{p_0}(f, f)
\leq
  C \norm{f}_{\lp{p_0}}
  \norm{f}_{L^1_w}^{p_0 \theta_4} \norm{f^{\frac{p_0}{2}}}_{H^s_{\gamma/2}}^{2(1-\theta_4)}
\leq
  C_{\epsilon} \, \norm{f}_{L^{p_0}}^{\frac{1}{\theta_4}}
  + \epsilon \norm{f^{\frac{p_0}{2}}}_{H^s_{\gamma/2}}^{2}, 
\end{align*}
where
\begin{align} \label{def:theta-2-q-0}
  \frac{1}{p_0 q_0} = \theta_4 + \frac{(1 - \theta_4)} {p_s}, 
\qquad
  \frac{1}{q_0} = 2+\frac{\gam}{3}-\frac{1}{p_0}, 
\qquad
  p_s = \frac{3 p_0}{3 - 2s}.
\end{align}
Note that $\theta_4$ satisfies
\begin{align*}
  \frac{1}{\theta_4}
= \frac{1 - \frac{1}{p_s}}{\frac{1}{p q_0} - \frac{1}{p_s}}
= p \frac{1 - \frac{1}{p_s}}{\frac{1}{q_0} - \frac{p}{p_s}}
> p \frac{1 - \frac{1}{p_s}}{1 - \frac{p}{p_s}}
> p. 
\end{align*}
Therefore, we only expect a short-time bound via this estimate. Inserting the bound of $I_{p_0}$ into the energy estimate, we have
\begin{align} \label{ineq:L-p-0-norm}
  \dv{t} \norm{f(t)}_{L^{p_0}}^{p_0} 
+ \frac{C_0}{2} \norm{f^{\frac{p_0}{2}}}_{\whp{s}{\gam/2}}^2
\leq
  C \vpran{\norm{f}_{L^{p_0}}^{p_0} + \norm{f}_{L^{p_0}}^{\frac{1}{\theta_4}}}.
\end{align}
By Gronwall's inequality, for $f_0 \in L^{p_0}$, there exists a short time $t_\ast$ such that 
\begin{align} \label{bound:L-p-0-short}
   \norm{f(t)}_{L^{p_0}} \leq C, 
\qquad
   t \in (0, t_\ast],
\end{align}
where $C$ depends on $\gam,s,D_0,E_0,p_0$ and $\norm{f_0}_{L^{p_0} \cap L^1_w}$. Given such $L^{p_0}$ for $f$, we can now generate any $L^p$-bound for $f$ on $(0,, t_\ast]$. By \eqref{bound:Ip-strong-2} in Lemma~\ref{lem: Ip-bound} again, we have
\begin{align*}
   I_{p}(f, f)
\leq
  C \norm{f}_{\lp{p_0}}
  \norm{f}_{L^1_w}^{p \theta_5} \norm{f^{\frac{p}{2}}}_{H^s_{\gamma/2}}^{2(1-\theta_5)}
\leq
  C_{\epsilon} + \epsilon \norm{f^{\frac{p}{2}}}_{\whp{s}{\gam/2}}^2,
\qquad
  \frac{1}{p q_0} = \theta_5 + \frac{1 - \theta_5}{p_s},
\end{align*}
where the last inequality makes use of the bound of $\norm{f}_{L^1_w}$ and $\norm{f}_{L^{p_0}}$. Therefore, the energy estimate reads
\begin{align} \label{energy:L-p-strong}
   \dv{t} \norm{f(t)}_{L^{p}}^{p} 
+ \frac{C_0}{2} \norm{f^{\frac{p}{2}}}_{\whp{s}{\gam/2}}^2
\leq
  C 
  \vpran{\norm{f}_{L^{p}}^{p} + 1},
\qquad
  t \in (0, t_\ast].
\end{align}
The generation of the $L^{p}$-norm for all $p \in \vpran{\frac{3}{3+\gam+2s},\infty}$ follows from the same interpolation as in~\eqref{bound:lower-H-s}, which gives
\begin{align} \label{bound:L-p-generation-strong}
   \dv{t} \norm{f(t)}_{L^p}^p
+ \frac{C_0}{4} \norm{f^{\frac{p}{2}}}_{\whp{s}{\gam/2}}^2
\leq
  C
\end{align}
and
\begin{equation*}
    \dv{t}\norm{f(t)}_{\lp{p}}^p + C_3 \frac{\norm{f(t)}_{\lp{p}}^{\frac{p}{\theta_3}}}{(1+T)^{p\frac{1-\theta_3}{\theta_3}}} 
\leq 
   C, \qquad \text{for all }  t \in (0,t_\ast].
\end{equation*}

\Ni Lastly, the $H^s_{\gamma/2}$-bound results from integrating~\eqref{bound:L-p-generation-strong} together with the bound of $\norm{f}_{L^p}$. 
\medskip

\Ni {\it ii)} Since $f_0 \in L^1 \cap L^p$, it also satisfies that $f_0 \in L^{p_0}$ for some $p_0 \in \left(\frac{3}{3+\gam+2s}, \frac{3}{3+\gam} \right)$. Then the propagation of the $L^{p_0}$ together with~\eqref{energy:L-p-strong} and the Gronwall inequality gives the propagation of the $L^p$-norm. 
\end{proof}

\section{$\lp{\infty}$ Theory}\label{sec;Linfty theory}

In this section, we use the De Giorgi level-set method to derive the $L^\infty$-bounds for the solutions to~\eqref{eq:Boltzmann}. Let $\chi_E$ be the characteristic function of $E \subseteq \RR^3$. Fix $K > 0$.  Denote 
\begin{align}\label{def:f-k}
    K_k\coloneqq K(1-2^{-k}),
\qquad
    f_k \coloneqq \left(f-K_k\right)\one{f\geq K_k},
\qquad
   k \in \mathbb{N}. 
\end{align}
For $t_\ast$ in the generation and propagation of $L^p$-norms in Theorem~\ref{thm:L-p-strong-soft}, let
\begin{align} \label{def:t-k}
  t_k := \frac{1}{2} t_\ast (1 - 2^{-(k+1)}) \in [t_\ast/8,  \, \, t_\ast], 
\qquad
  k \geq 0. 
\end{align}

The strategy will be to derive an energy estimate for $f_k$ via the weak formulation (\ref{eqn:weak-form}) with $\varphi = f_k^{p-1}$. The first step in this direction is given by the following lemma:
\begin{lemma} \cite[p. 9]{Alo19} \label{lem:Q-f-k}
Let $f \geq 0$ be sufficiently smooth and $f_k$ be as in~\eqref{def:f-k} for a fixed $K\geq0$. Then,
\begin{align} \label{est:Q-L-infty}
    \int_{\rn{3}} Q(f,f) f_k^{p-1}dv 
\leq 
   K I_{p-1}(f,f_k) + \frac{1}{p'} I_p(f,f_k) - \frac{1}{\max\{p,p'\}} J_p(f,f_k),
\end{align}
where $I_p(f, f_k)$ and $J_p(f, f_k)$ are given by \eqref{def:I-p}-\eqref{def:J-p}.
\end{lemma}

We recall in the next lemma the classical useful tool to break homogeneity in the interpolation. 

\begin{lemma}\label{lem:inhomog}
For any $\alpha\geq 0$ and $\beta \in [1, k]$,  one has
\begin{equation*}
   \one{f\geq K_k} 
\leq 
  \vpran{\frac{1}{2^\beta -1}}^\alpha 
  \vpran{\frac{2^k}{K}}^\alpha f_{k-\beta}^\alpha.
\end{equation*}
\end{lemma}

The $L^\infty$-bounds will be derived through the following decay of an energy functional $W_k$, which will be defined later. The bound in part (a) of Lemma~\ref{lem:W-k} has appeared in earlier works and part (b) is a generalization. For the convenience of the reader, we provide both proofs here. 

\begin{lemma}\label{lem:W-k}
(a) Let $a, b, C > 0$, $c > 1$, $k \in \mathbb{N}$ and $W_0 > 0$ be a constant. If 
\begin{equation}\label{eq:W-k}
   W_k 
\leq C 2^{ak} K^{-b} W_{k-1}^{c}
\quad \text{and} \quad
   K 
\geq 
   \vpran{C 2^{\frac{ac}{c-1}} W_0^{c-1}}^{\frac{1}{b}},
\end{equation}
then we have
\begin{equation*}
   W_k 
\leq 
   W_0\left(2^{-\frac{a}{c-1}}\right)^k,
\qquad
  k \geq 1. 
\end{equation*}

\smallskip
\Ni (b) More generally, if 
\begin{align*}
      W_k 
\leq 
   C 2^{ak} K^{-b} \vpran{W_{k-1}^{c_1} + W_{k-1}^{c_2}}, 
\qquad
   1 < c_1 \leq c_2
\end{align*}
with $K$ satisfying
\begin{align} \label{cond:K-general}
   K
\geq 
  \vpran{\max \left\{C 2^{c_2 + \frac{a c_1}{c_1-1}} W_0^{c_1+c_2-2}, \,\,
  C 2^{1+ \frac{a c_1}{c_1-1}} W_0^{c_1-1} \right\}}^{1/b}, 
\end{align}
then 
\begin{align} \label{bound:W-k-general}
   W_k 
\leq 
  W_0\left(2^{-\frac{a}{c_1-1}}\right)^k, 
\qquad
  k \geq 1. 
\end{align}
\end{lemma}
\begin{proof}
(a) Denote $\widetilde W_k = W_0\left(2^{-\frac{a}{c-1}}\right)^k$. The proof follows from showing that $\tilde W_k$ is a super-solution of~\eqref{eq:W-k} if $K$ satisfies the condition in ~\eqref{eq:W-k}. Indeed, 
\begin{align*}
 C 2^{ak} K^{-b} \widetilde W_{k-1}^c 
= C2^{ak}K^{-b}W_0^c\left(2^{-\frac{a}{c-1}}\right)^{c(k-1)} 
\leq \widetilde W_k, 
\qquad
  \widetilde W_0 = W_0.
\end{align*}
Hence $W_k \leq \widetilde W_k = W_0\left(2^{-\frac{a}{c-1}}\right)^k$. 

\smallskip
\Ni (b) 
If $W_0 \leq 1$, then we show by induction that 
\begin{align*}
   W_k 
\leq 1, 
\qquad \text{for all} \quad k \geq 1.
\end{align*}
First, for $k=1$, 
\begin{align*}
  W_1 
\leq 
  C 2^a K^{-b} (W_0^{c_1} + W_0^{c_2})
\leq
  C 2^{1+a} K^{-b} W_0^{c_1-1}
\leq 1,
\end{align*}
where we have used $W_0 \leq 1$ and $1 < c_1 \leq c_2$. At the induction step, suppose $W_j \leq 1$ for all $0 \leq j \leq k$. Then for any $0 \leq j \leq k$, 
\begin{align*}
  W_{j+1}
\leq
  C 2^{ak} K^{-b} \vpran{W_{j}^{c_1} + W_{j}^{c_2}}
\leq
  2 C 2^{ak} K^{-b} W_{j}^{c_1}, 
\qquad
  0 \leq j \leq k.  
\end{align*}
Therefore by the range of $K$, part (a) applies and we obtain that 
\begin{align*}
  W_{k+1} 
\leq 
  W_0 \vpran{2^{-\frac{a}{c_1 - 1}}}^{k+1} 
\leq 1. 
\end{align*}
If $W_0 > 1$, then let $\widehat W_k = \frac{W_k}{2 W_0}$. Such $\widehat W_k$ satisfies
\begin{align*}
   \widehat W_k 
\leq
  C (2W_0)^{c_2 - 1}2^{ak} K^{-b} \vpran{\widehat W_{k-1}^{c_1} + \widehat W_{k-1}^{c_2}}, 
\qquad
  \widehat W_0 = 1/2 < 1. 
\end{align*}
Therefore, by the previous case when $W_0 \leq 1$, we have
\begin{equation*} 
   \widehat W_k
\leq
  \widehat W_0 \vpran{2^{-\frac{a}{c_1 - 1}}}^k
\Longrightarrow
  W_k 
\leq 
  W_0 \vpran{2^{-\frac{a}{c_1 - 1}}}^k, 
\qquad
  k \geq 1.  \qedhere
\end{equation*}
\end{proof}

\subsection{$L^\infty$-generation for $\gam \leq -2s$} We start with the key estimates for $I_{p-1}(f, f_k)$ and $I_p(f, f_k)$. 

\begin{lemma}\label{lem:strong-soft}
Let $D_0,E_0>0$, $\gam\in(-3,-2s]$, $p\in \vpran{\frac{3}{3+\gam+2s},\infty}$. 
Suppose $f(t,v)$ is a sufficiently smooth solution to~\eqref{eq:Boltzmann} on $(0,T]\times\rn{3}$ with $f_0\in\mathcal{U}(D_0,E_0)\cap\lp{p} \cap L^1_w$ with $w$ sufficiently large and $B$ satisfying~\eqref{cond:B-1}-\eqref{cond:B-3}. Let $t_\ast$ be the time in Theorem~\ref{thm:L-p-strong-soft} for the $L^p$-propagation and the $L^{q'}$-generation of $f$ with $q'$ being the H\"{o}lder conjugate of $q$ defined in~\eqref{cond:theta-8-9} and~\eqref{cond:ell}. Let $I_p(f, f_k)$ and $J_p(f, f_k)$ be given by \eqref{def:I-p}-\eqref{def:J-p}. Then for any $t \in [\frac{t_\ast}{8}, t_\ast]$, 
\begin{align*}
  K I_{p-1}(t) + I_p(t)
&\leq
      C 2^{k(r - p + 1)} K^{-(r - p)}
  \norm{f_{k-1}}_{\lp{p}}^{p \theta_7}\norm{f_{k-1}^{p/2}}_{\whp{s}{\gam/2}}^{\frac{2p_s}{p}(1-\theta_6-\theta_7)}
\\
& \quad \, 
    + C 2^{k(\ell - p + 1)} K^{-(\ell - p)}
      \norm{f_{k-1}}_{L^p}^{\frac{p\theta_9}{q}}
  \norm{f_{k-1}^{p/2}}_{H^s_{\gamma/2}}^{\frac{2 p_s}{p q} (1-\theta_8-\theta_9)}, 
\end{align*}
where $\theta_6, \theta_7, r, \theta_8, \theta_9, \ell, q$ satisfy~\eqref{cond:theta-6-7}, \eqref{def:r}, \eqref{cond:theta-8-9} and~\eqref{cond:ell}, and the constant $C$ depends on $D_0, E_0, \gamma, \norm{f_0}_{L^p \cap L^1_w}, t_\ast, p$.
\end{lemma}

\begin{proof} By the first inequality in Corollary~\ref{cor:I-p}, we have
\begin{align}\label{eq:I-p-1}
    I_{p-1}(f,f_k) 
&\leq 
   C \int_{\rn{3}\times\rn{3}}f_*f_k^{p-1}\abs{v-v_*}^{\gam} dv_*\, dv \nonumber
\\
&= C \left(\int_{\abs{v-v_*}>1}f_*f_k^{p-1}\abs{v-v_*}^{\gam} dv_*\, dv 
+ \int_{\abs{v-v_*}
<1}f_*f_k^{p-1}\abs{v-v_*}^{\gam} dv_*\, dv\right) \nonumber
\\
&\leq 
  C \norm{f}_{\lp{1}}\int_{\rn{3}}f_k^{p-1}dv 
  + C \int_{\abs{v-v_*}<1}f_*f_k^{p-1}\abs{v-v_*}^{\gam} dv_*\, dv.
\end{align}

Choose $\theta_6, \theta_7 \in(0,1)$ such that 
\begin{align} \label{cond:theta-6-7}
   \theta_6+\theta_7<1,
\qquad
   p_s(1-\theta_6-\theta_7) < p, 
\qquad
  p\theta_7 + p_s(1-\theta_6-\theta_7) > p. 
\end{align} 
There is a range of $\theta_6, \theta_7$ where the conditions in~\eqref{cond:theta-6-7} can be satisfied. For example, we can choose $\theta_6$ sufficiently close to $0$ and $\theta_7 = 1 - \frac{p}{2p_s}$. Let
\begin{align} \label{def:r}
   r 
= \theta_6 + p\theta_7 + p_s(1-\theta_6-\theta_7). 
\end{align}
Then $p < r < p_s$. By applying Lemma \ref{lem:inhomog}, we get that
\begin{equation*}
   \int_{\rn{3}}f_k^{p-1}dv 
\leq 
  \left(\frac{2^k}{K}\right)^{r - p + 1}
  \int_{\rn{3}}f_{k-1}^{r}dv,
\qquad
  r > p.
\end{equation*}
By the H\"{o}lder's inequality and~\eqref{Sobolev}  we obtain the following bound:
\begin{align} \label{bound:f-k-1-1}
   \int_{\rn{3}}f_{k-1}^{r}dv 
&\leq 
   \norm{f_{k-1}}_{L^1_w}^{\theta_6}
     \norm{f_{k-1}}_{\lp{p}}^{p \theta_7}
     \norm{f_{k-1}}_{\wlp{p_s}{\gam/2}}^{p_s(1-\theta_6-\theta_7)}  \nn
\\
&\leq 
  C \norm{f_{k-1}}_{\wlp{1}{w}}^{\theta_6}\norm{f_{k-1}}_{\lp{p}}^{p \theta_7}\norm{f_{k-1}^{p/2}}_{\whp{s}{\gam/2}}^{\frac{2p_s}{p}(1-\theta_6-\theta_7)} \nn
\\
&\leq 
  C \norm{f_{k-1}}_{\lp{p}}^{p \theta_7}\norm{f_{k-1}^{p/2}}_{\whp{s}{\gam/2}}^{\frac{2p_s}{p}(1-\theta_6-\theta_7)}, 
\end{align}
where $w$ is large enough and $C$ depends on $\norm{f_0}_{L^1_{w}}$. Then the first term on the right-hand side of~\eqref{eq:I-p-1} satisfies
\begin{align*}
  C \norm{f}_{\lp{1}}\int_{\rn{3}}f_k^{p-1}dv
\leq 
  C \left(\frac{2^k}{K}\right)^{r-p+1}
  \norm{f_{k-1}}_{\lp{p}}^{p \theta_7}\norm{f_{k-1}^{p/2}}_{\whp{s}{\gam/2}}^{\frac{2p_s}{p}(1-\theta_6-\theta_7)}
\end{align*}
with $\theta_6, \theta_7, r$ satisfying~\eqref{cond:theta-6-7} and~\eqref{def:r}. 

The second term on the right-hand side of~\eqref{eq:I-p-1} is bounded as follows:
\begin{align} \label{bound:f-k-1-2}
  \int_{\abs{v-v_*}<1}f_*f_k^{p-1}\abs{v-v_*}^{\gam} dv_*\, dv
& = \int_{\mathbb{R}^6} f_*f_k^{p-1}\abs{v-v_*}^{\gam} \One_{|v-v_\ast| < 1}dv_*\, dv  \nn
\\
&= \int_{\RR^3} \vpran{f \ast |v|^\gamma \One_{|v|<1}} f_k^{p-1} dv \nn
\\
& \leq
  \norm{f \ast |v|^\gamma \One_{|v|<1}}_{L^{q'}}
  \norm{f_k^{p-1}}_{L^q} \nn
\\
& \leq
  C \norm{f}_{L^{q'}} \norm{f_k^{p-1}}_{L^q}
\leq
  C \norm{f_k^{p-1}}_{L^q},
\end{align}
by the $L^{q'}$-generation of the solution for $q' < \infty$ to be determined. Here $(q, q')$ are H\"{o}lder conjugates. Choose $\theta_8, \theta_9 \in(0,1)$, $q > 1$ and $\ell > p$ such that 
\begin{align} \label{cond:theta-8-9}
   \theta_8+\theta_9<1,
\qquad
   p_s(1-\theta_8-\theta_9) < p q, 
\qquad
  p\theta_9 + p_s(1-\theta_8-\theta_9) > p q, 
\end{align} 
and 
\begin{align} \label{cond:ell}
   \ell q = \theta_8 + p\theta_9 + p_s\vpran{1-\theta_8-\theta_9}, 
\qquad
   \ell > p. 
\end{align}
Similar as for $\theta_6, \theta_7$, we can find a range of $(\theta_8, \theta_9, \ell, q)$ such that~\eqref{cond:theta-8-9} and~\eqref{cond:ell} hold. For example, if we choose $q$ close enough to $1$, then $(\theta_8, \theta_9) = (\theta_6, \theta_7)$ will work.

By Lemma~\ref{lem:inhomog} and~\eqref{Sobolev}, we have
\begin{align} \label{bound:f-k-1-3}
  \norm{f_k^{p-1}}_{L^q}
&\leq
  \vpran{\frac{2^k}{K}}^{\ell - p + 1} \norm{f_{k-1}}_{L^{\ell q}}^{\ell}  \nn
\\
&\leq
  \vpran{\frac{2^k}{K}}^{\ell - p + 1} \norm{f_{k-1}}_{L^1_{w}}^{\frac{\theta_8}{q}}
  \norm{f_{k-1}}_{L^p}^{\frac{p\theta_9}{q}}
  \norm{f_{k-1}}_{L^{p_s}_{\gamma/2}}^{\frac{1-\theta_8-\theta_9}{q}} \nn
\\
&\leq
  C \vpran{\frac{2^k}{K}}^{\ell - p + 1} 
  \norm{f_{k-1}}_{L^p}^{\frac{p\theta_9}{q}}
  \norm{f_{k-1}^{p/2}}_{H^s_{\gamma/2}}^{\frac{2 p_s}{p q} (1-\theta_8-\theta_9)}. 
\end{align}
Therefore the second term on the right-hand side of~\eqref{eq:I-p-1} satisfies
\begin{align*}
  \int_{\abs{v-v_*}<1}f_*f_k^{p-1}\abs{v-v_*}^{\gam} dv_*\, dv
\leq 
  C
  \vpran{\frac{2^k}{K}}^{\ell - p + 1} 
  \norm{f_{k-1}}_{L^p}^{\frac{p\theta_9}{q}}
  \norm{f_{k-1}^{p/2}}_{H^s_{\gamma/2}}^{\frac{2 p_s}{p q} (1-\theta_8-\theta_9)}. 
\end{align*}
To summarize, we have
\begin{align*}
  I_{p-1}(f,f_k) 
&\leq 
  C \left(\frac{2^k}{K}\right)^{r - p + 1}
  \norm{f_{k-1}}_{\lp{p}}^{p \theta_7}\norm{f_{k-1}^{p/2}}_{\whp{s}{\gam/2}}^{\frac{2p_s}{p}(1-\theta_6-\theta_7)}
\\
& \quad \, 
    + C \left(\frac{2^k}{K}\right)^{\ell - p+1}
      \norm{f_{k-1}}_{L^p}^{\frac{p\theta_9}{q}}
  \norm{f_{k-1}^{p/2}}_{H^s_{\gamma/2}}^{\frac{2 p_s}{p q} (1-\theta_8-\theta_9)}, 
\end{align*}
where $r, \theta_6, \theta_7, \theta_8, \theta_9, \ell$ satisfy~\eqref{cond:theta-6-7}, \eqref{def:r}, \eqref{cond:theta-8-9} and~\eqref{cond:ell}. 

\smallskip
The bound for $I_p$ follows a similar line of estimates. In particular, replacing $p-1$ in~\eqref{eq:I-p-1} by $p$, we have
\begin{align}\label{eq:I-p}
    I_{p}(f,f_k) 
\leq 
  C \norm{f}_{\lp{1}}\int_{\rn{3}}f_k^{p}dv 
  + C \int_{\abs{v-v_*}<1}f_*f_k^{p}\abs{v-v_*}^{\gam} dv_*\, dv, 
\end{align}
where by~\eqref{bound:f-k-1-1},
\begin{align} \label{bound:f-k-p}
   \int_{\rn{3}}f_k^{p}dv
\leq
  \vpran{\frac{2^k}{K}}^{r-p}
  \int_{\rn{3}}f_{k-1}^{r}dv
\leq
  C \vpran{\frac{2^k}{K}}^{r-p}
  \norm{f_{k-1}}_{\lp{p}}^{p \theta_7}\norm{f_{k-1}^{p/2}}_{\whp{s}{\gam/2}}^{\frac{2p_s}{p}(1-\theta_6-\theta_7)},
\end{align}
with $\theta_6, \theta_7, r$ being the same parameters as in the bound for $I_{p-1}$. Moreover, by a similar estimate as in~\eqref{bound:f-k-1-2} and~\eqref{bound:f-k-1-3} and choosing the same parameters $\ell, \theta_8, \theta_9, q$, we have
\begin{align*}
  \int_{\abs{v-v_*}<1}f_*f_k^{p}\abs{v-v_*}^{\gam} dv_*\, dv
&\leq
  C \norm{f_k^p}_{L^{q}}
\leq
  C \vpran{\frac{2^k}{K}}^{\ell - p} \norm{f_{k-1}^{\ell}}_{L^{q}}
 \\
&\leq
  C \vpran{\frac{2^k}{K}}^{\ell - p} 
  \norm{f_{k-1}}_{L^p}^{\frac{p\theta_9}{q}}
  \norm{f_{k-1}^{p/2}}_{H^s_{\gamma/2}}^{\frac{2 p_s}{p q} (1-\theta_8-\theta_9)}.
\end{align*}
Therefore, $I_p$ is bounded as
\begin{align}
  I_p
&\leq 
  C \left(\frac{2^k}{K}\right)^{r - p}
  \norm{f_{k-1}}_{\lp{p}}^{p \theta_7}\norm{f_{k-1}^{p/2}}_{\whp{s}{\gam/2}}^{\frac{2p_s}{p}(1-\theta_6-\theta_7)} \nn
\\
& \quad \, 
    + C \left(\frac{2^k}{K}\right)^{\ell - p}
      \norm{f_{k-1}}_{L^p}^{\frac{p\theta_9}{q}}
  \norm{f_{k-1}^{p/2}}_{H^s_{\gamma/2}}^{\frac{2 p_s}{p q} (1-\theta_8-\theta_9)}, 
\end{align}
where $\theta_6, \theta_7, \theta_8, \theta_9$ satisfy~\eqref{cond:theta-6-7}, \eqref{cond:theta-8-9} and~\eqref{cond:ell}. The desired bound is the combination of the estimates for $I_p$ and $I_{p-1}$.
\end{proof}

\begin{remark} \label{Rmk:f-k-p}
The inequality~\eqref{bound:f-k-p} in the proof of Lemma~\ref{lem:strong-soft} also gives us the bound of $\norm{f_k}_{L^p}^p$ in terms of $\norm{f_{k-1}}_{\lp{p}}$ and $\norm{f_{k-1}^{p/2}}_{\whp{s}{\gam/2}}$. 
\end{remark}

The energy functional we will need is defined by
\begin{align} \label{def:W-k-strong-soft}
  W_k 
\coloneqq 
  \sup_{t\in[t_k, t_\ast]}\norm{f_k(t)}_{L^p}^p 
  + C\int_{t_k}^{t_\ast}\norm{f_k^{p/2}(\tau)}_{\whp{s}{\gam/2}}^2 d\tau,
\end{align}
where $t_\ast$ is the time period for the $L^p$-propagation and the $L^{q'}$-generation of $f$ with $q'$ being the H\"{o}lder conjugate of $q$ defined in~\eqref{cond:theta-8-9} and~\eqref{cond:ell} and $t_k$ is defined in~\eqref{def:t-k}.

We now show the generation of the $L^\infty$-norm. 
\begin{theorem}\label{thm:L-infty-strong}
Let $D_0, E_0, T>0$, $\gam\in(-3,-2s]$, $p\in \vpran{\frac{3}{3+\gam+2s}, \infty}$. 
Suppose $B$ satisfies~\eqref{cond:B-1}-\eqref{cond:B-3}. Suppose $f(t,v)$ is a sufficiently smooth solution to~\eqref{eq:Boltzmann} on $(0, T]\times\rn{3}$ with $f_0\in\mathcal{U}(D_0,E_0) \cap L^1_w \cap L^{p}$ where the weight $w$ is large enough.  Let $t_\ast$ be the time period in Theorem~\ref{thm:L-p-strong-soft} for the $L^p$-propagation and the $L^{q'}$-generation of $f$ with $q'$ being the H\"{o}lder conjugate of $q$ defined in~\eqref{cond:theta-8-9} and~\eqref{cond:ell}.  Then there are constants $C>0$ depending on $\gam,s,D_0,E_0, \norm{f_0}_{L^1_w \cap L^{p}}$ and $\alpha_6 > 0$ depending on $\gamma, s$ such that  
\begin{equation}\label{eqn;Linfty;generation small gam}
   \sup_{t\in[\frac{t_\ast}{2}, \,\, t_\ast]}\norm{f(t)}_{\lp{\infty}} 
\leq 
  C \vpran{t_\ast^{-\alpha_6} + 1}. 
\end{equation}
\end{theorem}
\begin{proof}
By Lemma~\ref{lem:Q-f-k}, the energy estimate for the level-set function $f_k$ reads
\begin{align*}
   \frac{1}{p}\dv{t}\norm{f_k(t)}_{\lp{p}}^p 
\leq KI_{p-1}(f,f_k) + \frac{1}{p'}I_p(f,f_k) - \frac{1}{\max\{p, p'\}}J_p(f,f_k).
\end{align*}
Combining~\eqref{bound:J-p} and Lemma~\ref{lem:strong-soft}, we have
\begin{align*}
   \dv{t}\norm{f_k(t)}_{\lp{p}}^p + C_4 \norm{f_k^{p/2}}_{H^s_{\gam/2}}^2
&\leq
  C 2^{k(r - p + 1)} K^{-(r - p)}
  \norm{f_{k-1}}_{\lp{p}}^{p \theta_7}\norm{f_{k-1}^{p/2}}_{\whp{s}{\gam/2}}^{\frac{2p_s}{p}(1-\theta_6-\theta_7)}
\\
& \quad \, 
    + C 2^{k(\ell - p + 1)} K^{-(\ell - p)}
      \norm{f_{k-1}}_{L^p}^{\frac{p\theta_9}{q}}
  \norm{f_{k-1}^{p/2}}_{H^s_{\gamma/2}}^{\frac{2 p_s}{p q} (1-\theta_8-\theta_9)}
  + C \norm{f_k}_{L^p}^p.
\end{align*}
Using the integration factor $e^{-Ct}$ and the boundedness of $t_\ast$, the last term can be absorbed and we obtain that 
\begin{align} \label{ineq:f-k-p-strong}
   \dv{t}\norm{f_k(t)}_{\lp{p}}^p 
   + C_5 \norm{f_k^{p/2}}_{H^s_{\gam/2}}^2
&\leq
  C 2^{k(r - p + 1)} K^{-(r - p)}
  \norm{f_{k-1}}_{\lp{p}}^{p \theta_7}
  \norm{f_{k-1}^{p/2}}_{\whp{s}{\gam/2}}^{\frac{2p_s}{p}(1-\theta_6-\theta_7)} \nn
\\
& \quad \, 
    + C 2^{k(\ell - p + 1)} K^{-(\ell - p)}
      \norm{f_{k-1}}_{L^p}^{\frac{p\theta_9}{q}}
  \norm{f_{k-1}^{p/2}}_{H^s_{\gamma/2}}^{\frac{2 p_s}{p q} (1-\theta_8-\theta_9)}. 
\end{align}
Integrating from $\xi \in [t_{k-1}, t_k]$ to $t \in [t_k, t_\ast]$, we have
\begin{align*}
& \quad \,
  \norm{f_k(t)}_{\lp{p}}^p
+ C_5 \int_{\xi}^t \norm{f_k^{p/2}(\tau)}_{H^s_{\gam/2}}^2  d\tau
\\
&\leq
  C 2^{k(r - p + 1)} K^{-(r - p)}
  \int_{\xi}^t \norm{f_{k-1}(\tau)}_{\lp{p}}^{p \theta_7}
  \norm{f_{k-1}^{p/2}(\tau)}_{\whp{s}{\gam/2}}^{\frac{2p_s}{p}(1-\theta_6-\theta_7)} d\tau
\\
& \quad \,
  + C 2^{k(\ell - p + 1)} K^{-(\ell - p)}
      \int_{\xi}^t \norm{f_{k-1}(\tau)}_{L^p}^{\frac{p\theta_9}{q}}
  \norm{f_{k-1}^{p/2}(\tau)}_{H^s_{\gamma/2}}^{\frac{2 p_s}{p q} (1-\theta_8-\theta_9)} d\tau
\\
& \quad \,
  + \norm{f_k(\xi)}_{\lp{p}}^p, 
\qquad
  r, \ell > p,
\end{align*}
which implies that 
\begin{align*}
& \quad \,
  \norm{f_k(t)}_{\lp{p}}^p
+ C_5 \int_{t_k}^t \norm{f_k^{p/2}(\tau)}_{H^s_{\gam/2}}^2  d\tau
\\
&\leq
  C 2^{k(r - p + 1)} K^{-(r - p)}
  \int_{t_{k-1}}^{t_\ast} \norm{f_{k-1}(\tau)}_{\lp{p}}^{p \theta_7}
  \norm{f_{k-1}^{p/2}(\tau)}_{\whp{s}{\gam/2}}^{\frac{2p_s}{p}(1-\theta_6-\theta_7)} d\tau
\\
& \quad \,
  + C 2^{k(\ell - p + 1)} K^{-(\ell - p)}
      \int_{t_{k-1}}^{t_\ast} \norm{f_{k-1}(\tau)}_{L^p}^{\frac{p\theta_9}{q}}
  \norm{f_{k-1}^{p/2}(\tau)}_{H^s_{\gamma/2}}^{\frac{2 p_s}{p q} (1-\theta_8-\theta_9)} d\tau
\\
& \quad \,
  + \norm{f_k(\xi)}_{\lp{p}}^p, 
\qquad
   r, \ell > p.
\end{align*}
Averaging the inequality above in $\xi$ on $[t_{k-1}, t_k]$ gives
\begin{align} \label{bound:f-k-p-xi}
& \quad \,
  \norm{f_k(t)}_{\lp{p}}^p
+ C_5 \int_{t_k}^t \norm{f_k^{p/2}(\tau)}_{H^s_{\gam/2}}^2  d\tau \nn
\\
&\leq
  C 2^{k(r - p + 1)} K^{-(r - p)}
  \int_{t_{k-1}}^{t_\ast} \norm{f_{k-1}(\tau)}_{\lp{p}}^{p \theta_7}
  \norm{f_{k-1}^{p/2}(\tau)}_{\whp{s}{\gam/2}}^{\frac{2p_s}{p}(1-\theta_6-\theta_7)} d\tau \nn
\\
& \quad \,
  + C 2^{k(\ell - p + 1)} K^{-(\ell - p)}
      \int_{t_{k-1}}^{t_\ast} \norm{f_{k-1}(\tau)}_{L^p}^{\frac{p\theta_9}{q}}
  \norm{f_{k-1}^{p/2}(\tau)}_{H^s_{\gamma/2}}^{\frac{2 p_s}{p q} (1-\theta_8-\theta_9)} d\tau \nn
\\
& \quad \,
  + \frac{1}{t_k - t_{k-1}}\int_{t_{k-1}}^{t_\ast}\norm{f_k(\xi)}_{\lp{p}}^p d\xi, 
\end{align}
where by the definition of $t_k$ and Remark~\ref{Rmk:f-k-p}, the last term satisfies
\begin{align*}
  \frac{1}{t_k - t_{k-1}}
  \int_{t_{k-1}}^{t_\ast}\norm{f_k(\xi)}_{\lp{p}}^p d\xi
& = \frac{2^{k+2}}{t_\ast}
      \int_{t_{k-1}}^{t_\ast}\norm{f_k(\xi)}_{\lp{p}}^p d\xi
\\
& \leq
   C 2^{k(r-p+1)} K^{-(r-p)}
  \int_{t_{k-1}}^{t_\ast} \norm{f_{k-1}}_{\lp{p}}^{p \theta_7}\norm{f_{k-1}^{p/2}}_{\whp{s}{\gam/2}}^{\frac{2p_s}{p}(1-\theta_6-\theta_7)}. 
\end{align*}
Applying it to~\eqref{bound:f-k-p-xi}, we have
\begin{align} \label{ineq:f-k-p}
& \quad \,
  \norm{f_k(t)}_{\lp{p}}^p
+ C_5 \int_{t_k}^t \norm{f_k^{p/2}(\tau)}_{H^s_{\gam/2}}^2  d\tau \nn
\\
&\leq
  C 2^{k(r - p + 1)} K^{-(r - p)}
  \int_{t_{k-1}}^{t_\ast} \norm{f_{k-1}(\tau)}_{\lp{p}}^{p \theta_7}
  \norm{f_{k-1}^{p/2}(\tau)}_{\whp{s}{\gam/2}}^{\frac{2p_s}{p}(1-\theta_6-\theta_7)} d\tau \nn
\\
& \quad \,
  + C 2^{k(\ell - p + 1)} K^{-(\ell - p)}
      \int_{t_{k-1}}^{t_\ast} \norm{f_{k-1}(\tau)}_{L^p}^{\frac{p\theta_9}{q}}
  \norm{f_{k-1}^{p/2}(\tau)}_{H^s_{\gamma/2}}^{\frac{2 p_s}{p q} (1-\theta_8-\theta_9)} d\tau. 
\end{align}
Taking the supremum of $t$ over $[t_k, t_\ast]$ in~\eqref{ineq:f-k-p} and using the definition of $W_k$ in~\eqref{def:W-k-strong-soft}, we obtain the recursive inequality
\begin{align*}
   W_k
 \leq
   C 2^{k(r - p + 1)} K^{-(r - p)} W_{k-1}^{\alpha_3}
   + C 2^{k(\ell - p + 1)} K^{-(\ell - p)} W_{k-1}^{\alpha_4}, 
\end{align*}
where 
\begin{align*}
  \alpha_3 
= \theta_7 + \frac{p_s}{p}(1-\theta_6-\theta_7) > 1, 
\qquad
  \alpha_4
= \frac{\theta_9}{q} + \frac{p_s}{pq} (1-\theta_8-\theta_9) > 1. 
\end{align*}
Therefore, if $W_0 < \infty$ and we choose
\begin{align*}
   K 
\geq 
    \vpran{\max \left\{C 2^{\alpha_4 + \frac{a \alpha_3}{\alpha_3-1}} W_0^{\alpha_3 + \alpha_4 - 2}, \,\,
  C 2^{1+ \frac{a \alpha_3}{\alpha_3-1}} W_0^{\alpha_3-1} \right\}}^{1/b} 
=: R_0^{1/b}, 
\end{align*}
where
\begin{align*}
   a = \max\{r-p+1, \,\, \ell - p + 1\}, 
\qquad
   b 
= \begin{cases}
    \min\{r - p, \ell - p\}, & \text{if \,\,} R_0 > 1, \\[3pt]
    \max\{r - p, \ell - p\}, & \text{if \,\,} R_0 < 1,
   \end{cases}
\end{align*}
then by Lemma~\ref{lem:W-k} part (b),  we have $W_k \to 0$, which implies that $\norm{f}_{L^\infty} \leq K$ on $[t_\ast/2, \, t_\ast]$ since $t_k \to t_\ast/2$. The condition $W_0 < \infty$ follows from the energy estimate. This is because $f_k = f$ for $k = 0$ and 
\begin{align*}
   W_0
 = \sup_{t\in \left[\frac{t_\ast}{2}, t_\ast \right]}\norm{f(t)}_{L^p}^p 
  + C\int_{t_\ast/2}^{t_\ast}\norm{f^{p/2}(\tau)}_{\whp{s}{\gam/2}}^2 d\tau.
\end{align*}
Therefore, by Theorem~\ref{thm:L-p-strong-soft}, we have
\begin{align} \label{bound:W-0}
  W_0 
\leq
  C (t_\ast^{-\alpha_5} + 1)
< \infty
\end{align}
for some $\alpha_5 > 0$. Here $C$ is independent of $t_\ast$ since all the dependence on $t_\ast$ has been combined into $t_\ast^{-\alpha_5}$. Therefore,  
\begin{align*}
   \norm{f}_{L^\infty}
\leq
  K 
\leq 
  C \vpran{\max \left\{t_\ast^{-\alpha_5 (\alpha_3 + \alpha_4 - 2)}, \,\, t_\ast^{-\alpha_5 (\alpha_3 - 1)} \right\} + 1}, 
\qquad
  t \in [t_\ast/2, \,\, t_\ast].
\end{align*}
If we take $t_\ast < 1$, then 
\begin{equation*}
  \sup_{t\in[\frac{t_\ast}{2}, \,\, t_\ast]}\norm{f(t)}_{\lp{\infty}} 
\leq 
  C \vpran{t_\ast^{-\alpha_5 (\alpha_3 + \alpha_4 - 2)} + 1}.  \qedhere
\end{equation*}
\end{proof}

\subsection{$\lp{\infty}$ Generation for $\gam>-2s$}\label{sec:Linfty-generation-weak}

The proof for $\gamma > -2s$ is similar and more straightforward than for the case when $\gamma \leq -2s$. We use similar notation
\begin{align*}
   K_k := K(1-2^{-k}),
\qquad
   f_k := (f-K_k)\one{f\geq K_k}, 
\qquad
  t_k := t_\ast (1 - 2^{-(k+1)}) \in [\tfrac{t_\ast}{2},  \, \, t_\ast], 
\end{align*}
and
\begin{align*}
   W_k
=   \sup_{t\in[t_k, T]}\norm{f_k(t)}_{L^p}^p 
  + C\int_{t_k}^{T}\norm{f_k^{p/2}(\tau)}_{\whp{s}{\gam/2}}^2 d\tau.
\end{align*}

The main result in this case states
\begin{theorem}\label{thm-Linfty-weak-soft}
Let $D_0,E_0,T>0$ be fixed. Suppose $B$ satisfies~\eqref{cond:B-1}-\eqref{cond:B-3} with $\gam\in(-2s, 0)$. If $f(t,v)$ is a sufficiently smooth solution to~\eqref{eq:Boltzmann} on $(0,T]\times\rn{3}$ with $f_0(v)\in\mathcal{U}(D_0,E_0)\cap L^1_w$ where the weight $w$ is large enough, then there are constants $C, \alpha_7 > 0$ where $C$ depends only on $\gam,s,D_0,E_0, T$ and $\norm{f_0}_{L^1_w}$ and $\alpha_7$ depends on $\gamma$ and $s$ such that
\begin{align} \label{bound:Linfty-weak}
  \sup_{t\in[t_\ast/2, \,T]} \norm{f(t)}_{\lp{\infty}} 
\leq 
   C \vpran{t_\ast^{-\alpha_7} + 1},
\qquad
  \forall t_\ast \in (0, T).
\end{align}
\end{theorem}
\begin{proof}
By Theorem~\ref{thm:Lp-weak-soft}, the generation of the $L^p$-norm holds for all $p \geq 1$ in the case of $\gamma > -2s$. Therefore, we can simply choose $p=2$. By multiplying~\eqref{eq:Boltzmann} by $f_k$, integrating over $v\in\rn{3}$ and by lemma~\ref{est:Q-L-infty} we have
\begin{equation*}
\frac{1}{2}\dv{t}\norm{f_k(t)}_{\lp{2}}^2 \leq KI_1(f,f_k) + \frac{1}{2}I_2(f,f_k) - \frac{1}{2}J_2(f,f_k).
\end{equation*}
By Lemma~\ref{lem:J-p} and~\eqref{bound:Ip-weak-soft} in Lemma~\ref{lem: Ip-bound}, we obtain
\begin{align} \label{eq:f-k-weak}
   \frac{1}{2}\dv{t}\norm{f_k(t)}_{\lp{2}}^2 
+ \frac{C_0}{2} \norm{f_k}_{\whp{s}{\gam/2}}^2 
\leq 
   KI_1(f,f_k) + C \norm{f_k}_{\wlp{2}{\gam/2}}^2 
\leq 
   KI_1(f,f_k) + C \norm{f_k}_{\lp{2}}^2,
\end{align}
where 
\begin{align*}
   I_1(f, f_k)
= \int_{\rn{3}\times\rn{3}}\int_{\sn{2}} g_* (f'-f)B(v-v_*,\sigma)d\sigma\, dv_*\, dv. 
\end{align*}
Similar as Corollary~\ref{cor:I-p}, by the Cancellation Lemma, we have
\begin{align*}
  \abs{I_1}
&\leq
  C \int_{\rn{3}\times\rn{3}} f_\ast f_k \abs{v-v_*}^\gam dv_*\, dv,
\\
& \leq
  C \vpran{\int_{\abs{v-v_*}\leq 1}f_*f_k\abs{v-v_*}^\gam dv_*\, dv + \int_{\abs{v-v_*} > 1}f_*f_k\abs{v-v_*}^\gam dv_*\, dv}
\\
&\leq 
  C\int_{\abs{v-v_*}\leq 1}f_*f_k\abs{v-v_*}^\gam dv_*\, dv 
  + C\norm{f}_{\lp{1}}\norm{f_k}_{\lp{1}}. 
\end{align*}
Take $1<q<\frac{3}{\abs{\gam}}$ and let $q' > 1$ be the H\"{o}lder conjugate of $q$. Then
\begin{align*}
\int_{\abs{v-v_*}\leq 1} f_*f_k\abs{v-v_*}^\gam dv_*\, dv &\leq \norm{f}_{\lp{q'}}\norm{f_k*\abs{\cdot}^\gam\one{\abs{\cdot}\leq1}}_{\lp{q}} 
\\
&\leq 
  \norm{f}_{\lp{q'}}\norm{\abs{\cdot}^\gam\one{\abs{\cdot}\leq1}}_{\lp{q}}\norm{f_k}_{\lp{1}} 
\\
&= C\norm{f}_{\lp{q'}}\norm{f_k}_{\lp{1}}.
\end{align*}
By the generation of the $L^{q'}$-norm in Theorem~\ref{thm:Lp-weak-soft}, we obtain the bound
\begin{align*}
  |I_1 (t)|
\leq
  C_6 \vpran{\norm{f}_{\lp{q'}} + \norm{f}_{L^1}} \norm{f_k}_{\lp{1}}
\leq
  C_7 (t_\ast^{-\alpha_1} + 1) \norm{f_k}_{\lp{1}}
\leq
  C \norm{f_k}_{\lp{1}}, 
\qquad
  t \in [t_\ast, T],
\end{align*}
where $C$ depends on $T$, $\gamma$ and $t_\ast$ but is independent of $K$ or $k$. Applying the estimate of $I_1$ to~\eqref{eq:f-k-weak} and using the integrating factor $e^{-Ct}$ to absorb the last term on the right-hand side of~\eqref{eq:f-k-weak} gives
\begin{align} \label{bound:f-k-weak-1}
   \frac{1}{2}\dv{t}\norm{f_k(t)}_{\lp{2}}^2 
+ C_8 \norm{f_k}_{\whp{s}{\gam/2}}^2 
\leq 
  C K \norm{f_k}_{\lp{1}}. 
\end{align}
If we choose $\beta = 1$ in Lemma~\ref{lem:inhomog}, then
\begin{align*}
  \norm{f_k}_{\lp{1}}
= \int_{\RR^3} f_k \, \chi_{f \geq K_k} dv
\leq
  C \vpran{\frac{2^k}{K}}^\alpha 
  \int_{\RR^3} f_{k-1}^{1+\alpha} dv. 
\end{align*}
Choose $\alpha > 1$ and $\theta_{10}, \theta_{11}$ such that
\begin{align*}
  1 + \alpha = \theta_{10} + 2 \theta_{11} + \frac{6}{3 - 2s} (1 - \theta_{10} - \theta_{11}), 
\qquad
    \theta_{10}, \, \theta_{11} \in (0, 1),
\end{align*}
and 
\begin{align*}
\qquad
  0 < \theta_{10} + \theta_{11}  < 1, 
\qquad
  \theta_{11} + \frac{3}{3 - 2s} (1 - \theta_{10} - \theta_{11}) > 1, 
\qquad
  \frac{3}{3 - 2s} (1 - \theta_{10} - \theta_{11}) \leq 1.
\end{align*}
Such $(\alpha, \theta_{10}, \theta_{11})$ exist. For example, we can choose $\theta_{10}$ close enough to zero and $\theta_{11} = \frac{2s}{3}$. Then by interpolation, it holds that 
\begin{align*}  
  \int_{\RR^3} f_{k-1}^{1+\alpha} dv
&\leq
  \norm{f_{k-1}}_{L^1_w}^{\theta_{10}}
  \norm{f_{k-1}}_{L^2}^{2 \theta_{11}}
  \norm{f_{k-1}}_{L^{\frac{6}{3-2s}}_{\gamma/2}}^{\frac{6}{3-2s} (1 - \theta_{10} - \theta_{11})}
\\
&\leq
  \norm{f_{k-1}}_{L^1_w}^{\theta_{10}}
  \norm{f_{k-1}}_{L^2}^{2\theta_{11}}
  \norm{f_{k-1}}_{H^s_{\gamma/2}}^{\frac{6}{3 - 2s} (1 - \theta_{10} - \theta_{11})}, 
\end{align*}
where the weight $w$ is large enough. Inserting such bound into~\eqref{bound:f-k-weak-1}, we obtain that
\begin{align} \label{ineq:f-k-2-weak}
  \frac{1}{2}\dv{t}\norm{f_k(t)}_{\lp{2}}^2 
+ C_8 \norm{f_k}_{\whp{s}{\gam/2}}^2
\leq
  C 2^{\alpha k} K^{-\alpha+1}
  \norm{f_{k-1}}_{L^2}^{2\theta_{11}}
  \norm{f_{k-1}}_{H^s_{\gamma/2}}^{\frac{6}{3 - 2s} (1 - \theta_{10} - \theta_{11})},
\end{align}
for any $t \in [t_\ast, T]$. Since \eqref{ineq:f-k-2-weak} is in a similar format as~\eqref{ineq:f-k-p-strong}, the rest follows the same line of argument as for Theorem~\ref{thm:L-infty-strong} and we omit the details to avoid repetition. 
\end{proof}

\bibliographystyle{plain}
\bibliography{references}

\end{document}